\documentclass[11pt,reqno]{amsart}

\usepackage[a4paper,margin=3cm]{geometry}
\usepackage{amsmath,amssymb,amsthm,mathtools}
\usepackage{microtype}
\usepackage{enumitem}
\usepackage{xcolor}
\usepackage{tikz}
\usetikzlibrary{arrows.meta,positioning,calc}
\usepackage[
  colorlinks=true,
  linkcolor=blue,
  citecolor=blue,
  urlcolor=blue
]{hyperref}
\usepackage[capitalise,noabbrev]{cleveref}

\hypersetup{
  pdftitle={Factorial Residues Modulo a Prime: Beyond the Square-Root Bound},
  pdfauthor={Xiyu Hu},
  pdfsubject={A lower bound for the value set of factorials modulo a prime},
  pdfkeywords={factorial residues, value sets, finite fields, incidence geometry}
}

\numberwithin{equation}{section}
\allowdisplaybreaks
\setlist{itemsep=0pt,topsep=0.35em,parsep=0pt,partopsep=0pt}

\newtheorem{theorem}{Theorem}[section]
\newtheorem{lemma}[theorem]{Lemma}

\theoremstyle{remark}
\newtheorem{remark}[theorem]{Remark}

\newcommand{\F}{\mathbb{F}}
\newcommand{\cL}{\mathcal{L}}
\newcommand{\cN}{\mathcal{N}}
\newcommand{\cP}{\mathcal{P}}
\newcommand{\I}{\mathcal{I}}

\title[Factorial residues beyond the square-root bound]
{Factorial Residues Modulo a Prime:\\Beyond the Square-Root Bound}
\author{Xiyu Hu}
\date{July 2026}


\begin{document}

\begin{abstract}
For a prime \(p\), let
\[
  A_p=\{k!\pmod p:1\leq k<p\}.
\]
We prove
\[
  |A_p|\gg p^{8/15}.
\]
This improves the general lower bound
\((\sqrt{2}-o(1))p^{1/2}\).
The proof starts from the identity
\[
  (n+2)!=(n+1)!+\frac{((n+1)!)^2}{n!}
\]
in \(\F_p\), which produces many incidences for a family of
fractional-linear maps.  After Cauchy--Schwarz, the transition maps between
two members of this family become affine lines, with multiplicity at most
two.  The Cartesian-product point-line incidence theorem of Stevens and
de Zeeuw then yields the exponent \(8/15\).
\end{abstract}

\maketitle

\section{Introduction}

Let \(p\) be a prime and write
\[
  A_p=\{k!\pmod p:1\leq k<p\}\subset\F_p^\times.
\]
Erd\H{o}s and Graham asked whether
\begin{equation}\label{eq:erdos-problem}
  |A_p|\sim (1-e^{-1})p
  \qquad (p\to\infty);
\end{equation}
see \cite[p.~96]{ErdosGraham1980}.  It is also listed as Erd\H{o}s
Problem~\#478 on Thomas F. Bloom's \emph{Erd\H{o}s Problems} website
\cite{Bloom478}.  The constant in \eqref{eq:erdos-problem} is the one
suggested by a random-mapping heuristic.  Even a positive-density lower
bound remains out of reach.

There is an elementary square-root lower bound.  Since \(0!=1=1!\) as
residue classes and
\[
  \frac{k!}{(k-1)!}=k\qquad(1\leq k<p),
\]
one has
\[
  A_p/A_p=\F_p^\times,
\]
and hence \(|A_p|\geq\sqrt{p-1}\).  Grebennikov, Sagdeev, Semchankau and
Vasilevskii strengthened this lower bound to
\[
  |A_p|\geq(\sqrt{2}-o(1))p^{1/2};
\]
see \cite{GrebennikovEtAl2024}.  Average questions concerning the
distribution of factorial residues were studied by Klurman and Munsch
\cite{KlurmanMunsch2017}; see also \cite{BanksEtAl2005} for earlier work on
the value set.

Our main result improves the exponent.

\begin{theorem}\label{thm:main}
For every prime \(p\),
\[
  |A_p|\gg p^{8/15},
\]
where the implied constant is absolute.
\end{theorem}

A formally related value-set problem concerns the self-power map
\[
  S_p=\{x^x\pmod p:1\leq x<p\}.
\]
Crocker proved the square-root-scale estimate \cite{Crocker1969}:
\[
  |S_p|\geq \left\lfloor\sqrt{\frac{p-1}{2}}\right\rfloor.
\]
Subsequent work has obtained substantially finer information about individual
fibres.  An important methodological precursor for the large-order regime is
the work of Bourgain and Shparlinski \cite{BourgainShparlinski2008} on
exponential sums with consecutive modular roots.  Writing
\(N_p(a)=\#\{1\leq x<p:x^x\equiv a\pmod p\}\), Balog, Broughan and
Shparlinski \cite{BalogBroughanShparlinski2011} note that their large-order
argument is similar to this approach; they proved the uniform bound
\(N_p(a)\leq p^{12/13+o(1)}\) and a nontrivial estimate for the collision
sum \(\sum_aN_p(a)^2\).  Cilleruelo and Garaev
\cite{CillerueloGaraev2016} strengthened several fibre estimates, including
\(N_p(1)\leq p^{27/82+o(1)}\), and obtained further bounds after sorting
\(a\) by its multiplicative order.

The proofs make essential use of the cyclic structure of
\(\F_p^\times\), of order \(p-1\).  After grouping by multiplicative order
and taking discrete logarithms, they split the solutions according to
divisors of \(p-1\).  For small orders, interval structure controls additive
growth while subgroups constrain multiplicative growth, enabling a
sum--product estimate.  For large orders, M\"obius inversion controls the
lifts of reduced residue classes between moduli dividing \(p-1\), and sparse
polynomial estimates supply the saving.  Cilleruelo and Garaev similarly use
intersections of intervals with multiplicative subgroups and exponential
sums over those subgroups.  The organisation is elementary in spirit, but
the exponent savings depend on deeper estimates.

These results give strong control of single fibres and collision moments, but
they do not yield \(|S_p|\geq p^{1/2+\delta}\) for any fixed \(\delta>0\).
Thus, before the present theorem, the general value-set bounds for both
\(S_p\) and \(A_p\) remained at the square-root scale.  The new input here is
specific to consecutive factorials: it turns a second moment into incidences
with a low-multiplicity family of affine lines.  We do not know an analogous
incidence-geometric linearisation for the self-power map.

The argument is based on three consecutive factorials.  If
\[
  x=n!,\qquad a=(n+1)!,\qquad z=(n+2)!,
\]
then
\[
  z=a+\frac{a^2}{x}.
\]
For \(a\in\F_p^\times\), define
\[
  T_a(x)=a+\frac{a^2}{x}.
\]
The factorial sequence supplies \(\gg p\) pairs \((x,a)\in A_p^2\) for
which \(T_a(x)\in A_p\).  The key algebraic observation is that
\[
  T_b\circ T_a^{-1}
\]
is affine.  Consequently, after Cauchy--Schwarz, the relevant second moment
is bounded by incidences between \(A_p\times A_p\) and a set of affine
lines.  We first recall the incidence estimate and then carry out this
linearisation.

Throughout the paper, all implied constants are absolute.

\begin{figure}[t]
\centering
\begin{tikzpicture}[
  x=1cm,
  y=1cm,
  font=\scriptsize,
  stage/.style={
    draw=black!68,
    line width=0.48pt,
    rounded corners=2.5pt,
    fill=black!1.5,
    align=center,
    inner xsep=4pt,
    inner ysep=3.2pt
  },
  merge/.style={
    draw=black!82,
    line width=0.62pt,
    rounded corners=2.5pt,
    fill=black!3,
    align=center,
    inner xsep=4pt,
    inner ysep=3.2pt
  },
  flow/.style={
    -{Latex[length=1.8mm,width=1.15mm]},
    line width=0.52pt,
    draw=black!78,
    shorten <=1.2pt,
    shorten >=1.2pt
  },
  bluefield/.style={draw=blue!62!black},
  redfield/.style={draw=red!68!black}
]
  \definecolor{fieldblue}{RGB}{38,86,143}
  \definecolor{fieldred}{RGB}{157,61,48}

  \node[font=\small\bfseries] at (3.25,7.65)
    {The algebraic-to-geometric passage};
  \node[font=\small\bfseries] at (10.9,7.65)
    {An exact finite-field snapshot (\(p=13\))};
  \draw[densely dashed,black!30] (7.18,0.1) -- (7.18,7.3);

  \node[stage,text width=5.8cm] (factorials) at (3.25,6.55)
    {\textsc{\scriptsize Factorial transitions}\\[-1pt]
     \(\displaystyle T_a(x)=a+\frac{a^2}{x}\)\\[-1pt]
     \(\displaystyle
       x=n!,\quad a=(n+1)!,\quad T_a(x)=(n+2)!\in A\)\\[-1pt]
     \(\displaystyle \mathcal N(A)\geq p-2\)};

  \node[
    stage,
    text width=5.05cm,
    below=4.8mm of factorials
  ] (pairing)
    {\textsc{\scriptsize Pair a common input}\\[-1pt]
     \(\displaystyle x,a,b\in A\)\\[-1pt]
     \(\displaystyle
       \bigl(u,v\bigr)=\bigl(T_a(x),T_b(x)\bigr)\in A^2\)};

  \node[
    stage,
    text width=2.4cm,
    minimum height=1.18cm,
    below=5.2mm of pairing,
    xshift=-1.78cm
  ] (point)
    {\textsc{\tiny Output point}\\[-1pt]
     \(\displaystyle P_{x;a,b}=(u,v)\)\\[-1pt]
     \(\displaystyle {}\in A\times A\)};

  \node[
    stage,
    text width=3.45cm,
    minimum height=1.18cm,
    below=5.2mm of pairing,
    xshift=1.78cm
  ] (line)
    {\textsc{\tiny Parameter line}\\[-1pt]
     \(\displaystyle r=\frac ba\)\\[-1pt]
     \(\displaystyle \ell_{a,b}:\ v=r^2u+ar(1-r)\)};

  \coordinate (branchmid) at ($(point.south)!0.5!(line.south)$);
  \node[
    merge,
    text width=5.4cm,
    below=5.2mm of branchmid
  ] (incidence)
    {\(\displaystyle P_{x;a,b}\in\ell_{a,b}\)\\[-1pt]
     because \(\displaystyle v=T_b\!\left(T_a^{-1}(u)\right)\)\\[-1pt]
     \(\displaystyle
       r^2=\operatorname{slope}(\ell)
       \ \Longrightarrow\ r\in\{\rho,-\rho\}\)\\[-1pt]
     \(\displaystyle
       |\mathcal L_A|\leq m^2,\qquad
       \operatorname{mult}(\ell)\leq2\ \text{off }v=u\)};

  \draw[flow] (factorials.south) -- (pairing.north);
  \draw[flow]
    ($(pairing.south west)!0.30!(pairing.south east)$)
    -- (point.north);
  \draw[flow]
    ($(pairing.south west)!0.70!(pairing.south east)$)
    -- (line.north);
  \draw[flow] (point.south)
    -- ($(incidence.north west)!0.30!(incidence.north east)$);
  \draw[flow] (line.south)
    -- ($(incidence.north west)!0.70!(incidence.north east)$);

  \node[anchor=west,font=\scriptsize] at (7.58,7.12)
    {\(\displaystyle
      A_{13}=\{1,2,3,5,6,7,9,11,12\}\subset\F_{13}^{\times}
      \quad(0\notin A_{13})\)};
  \node[anchor=west,font=\scriptsize] at (7.58,6.76)
    {\(\displaystyle
      \mathcal P_{13}=A_{13}\times A_{13}\subset\F_{13}^{2},
      \qquad 0\in\F_{13}\)};

  \begin{scope}[shift={(8.35,1.75)},x=0.325cm,y=0.325cm]
    \draw[step=1,black!8,line width=0.25pt] (0,0) grid (12,12);
    \draw[->,line width=0.55pt] (-0.28,0) -- (12.72,0)
      node[right=2pt] {\(u\)};
    \draw[->,line width=0.55pt] (0,-0.28) -- (0,12.72)
      node[above=2pt] {\(v\)};

    \foreach \t in {0,4,8,12} {
      \draw[black!55,line width=0.35pt] (\t,0.12) -- (\t,-0.12)
        node[below=1.3pt,font=\tiny] {\(\t\)};
      \draw[black!55,line width=0.35pt] (0.12,\t) -- (-0.12,\t)
        node[left=1.3pt,font=\tiny] {\(\t\)};
    }

    \foreach \u in {1,2,3,5,6,7,9,11,12} {
      \foreach \v in {1,2,3,5,6,7,9,11,12} {
        \fill[black!32] (\u,\v) circle (0.062);
      }
    }

    \foreach \q in {
      (0,7),(1,11),(2,2),(3,6),(4,10),(5,1),(6,5),
      (7,9),(8,0),(9,4),(10,8),(11,12),(12,3)
    } {
      \draw[fieldblue,line width=0.55pt,fill=white] \q circle (0.105);
    }
    \foreach \q in {(12,3),(1,11),(6,5),(11,12),(5,1),(7,9)} {
      \fill[fieldblue] \q circle (0.16);
      \draw[white,line width=0.28pt] \q circle (0.16);
    }

    \foreach \q in {
      (0,4),(1,3),(2,2),(3,1),(4,0),(5,12),(6,11),
      (7,10),(8,9),(9,8),(10,7),(11,6),(12,5)
    } {
      \draw[fieldred,line width=0.55pt,fill=white] \q circle (0.105);
    }
    \foreach \q in {(11,6),(3,1),(2,2),(12,5),(6,11)} {
      \fill[fieldred] \q circle (0.16);
      \draw[white,line width=0.28pt] \q circle (0.16);
    }
    \path[fill=fieldblue]
      (2,2) ++(0,0.16) arc (90:270:0.16) -- cycle;
    \path[fill=fieldred]
      (2,2) ++(0,-0.16) arc (-90:90:0.16) -- cycle;
    \draw[white,line width=0.28pt] (2,2) circle (0.16);

    \draw[black!70,line width=0.38pt] (0,0) rectangle (12,12);
    \draw[black!65,-{Latex[length=1.25mm,width=0.8mm]},line width=0.38pt]
      (6.15,4.75) -- (6,5);
    \node[
      anchor=south west,
      fill=white,
      fill opacity=0.9,
      text opacity=1,
      inner sep=1pt,
      font=\tiny
    ] at (6.18,4.78)
      {\(x=3:\ P_{3;3,6}=(6,5)\)};
  \end{scope}

  \node[
    anchor=west,
    text=fieldblue,
    font=\scriptsize,
    align=left,
    text width=6.0cm,
    inner sep=0pt
  ] at (8.15,1.27)
    {\(\circ\ \ell_B=\ell_{3,6}=\ell_{1,11}:\ v=4u+7\)};
  \node[
    anchor=west,
    text=fieldred,
    font=\scriptsize,
    align=left,
    text width=6.0cm,
    inner sep=0pt
  ] at (8.15,0.90)
    {\(\circ\ \ell_R=\ell_{5,12}=\ell_{12,5}:\ v=-u+4\)};

  \draw[flow] (incidence.east)
    .. controls (7.32,1.82) and (7.55,2.55) ..
    (8.18,3.21);

  \node[
    draw=black!82,
    line width=0.64pt,
    rounded corners=2.5pt,
    fill=black!1.5,
    align=center,
    text width=14.2cm,
    inner sep=5pt
  ] at (7.08,-0.72)
  {\(\displaystyle
    p^2\ll \mathcal N(A)^2
    \leq m\bigl(m^2+2\I(\mathcal P,\mathcal L_A)\bigr)
    \ll m\cdot m^{11/4}=m^{15/4}\)\\[3pt]
   \(\displaystyle
    \Longrightarrow\qquad
    |A_p|=m\gg p^{8/15}>p^{1/2}.
   \)};
\end{tikzpicture}
\caption{The structural passage behind the proof.  The left-hand diagram
shows how two evaluations with a common input produce a point and how
eliminating that input produces the line containing it.  All coordinates
and equations in the right-hand panel are over \(\F_{13}\).  The grey
dots form \(\mathcal P_{13}=A_{13}\times A_{13}\), while the blue and red
circular markers display the full \(13\)-point sets
\(\ell_B=\{(u,v):v=4u+7\}\) and
\(\ell_R=\{(u,v):v=-u+4\}\), respectively.  A filled circle marks an
incidence supplied by an admissible input \(x\in A_{13}\); the disc at
their common incidence \((2,2)\) is split between the two colours.  The
equalities below the plot record the two parameter pairs representing
each sample line.}
\label{fig:incidence-linearisation}
\end{figure}

\section{A Cartesian-product incidence bound}

For a set of points \(\cP\subseteq\F_p^2\) and a set of affine lines
\(\cL\), let
\[
  \I(\cP,\cL)
  =|\{(q,\ell)\in\cP\times\cL:q\in\ell\}|.
\]
We use the following theorem of Stevens and de Zeeuw
\cite[Theorem~4]{StevensDeZeeuw2017}.

\begin{theorem}[Stevens--de Zeeuw]\label{thm:sdz}
There are absolute constants \(c_0,C_0>0\) with the following property.
Let \(X,Y\subseteq\F_p\), with \(|X|=x\leq y=|Y|\), and let \(\cL\) be a
set of \(N\) distinct affine lines.  Suppose that
\[
  xy^2\leq N^3
  \qquad\text{and}\qquad
  xN\leq c_0p^2.
\]
Then
\[
  \I(X\times Y,\cL)
  \leq C_0\bigl(x^{3/4}y^{1/2}N^{3/4}+N\bigr).
\]
\end{theorem}

We shall apply \cref{thm:sdz} with \(X=Y=A_p\) and \(N=|A_p|^2\).

\section{Many factorial transitions}

Let \(A=A_p\) and set \(m=|A|\).  For \(a,x\in\F_p^\times\), define
\[
  T_a(x)=a+\frac{a^2}{x}.
\]
Consider the transition count
\[
  \cN(A)=|\{(x,a)\in A^2:T_a(x)\in A\}|.
\]

\begin{lemma}\label{lem:many-transitions}
For every odd prime \(p\),
\[
  \cN(A_p)\geq p-2.
\]
\end{lemma}

\begin{proof}
For each integer \(n\) with \(0\leq n\leq p-3\), put
\[
  x=n!,\qquad a=(n+1)!,\qquad z=(n+2)!
\]
in \(\F_p^\times\).  Although the definition of \(A_p\) starts with
\(1!\), the residue \(0!=1\) belongs to \(A_p\) because \(0!=1!\).
Hence \(x,a,z\in A_p\).  Moreover,
\[
  \frac{a}{x}=n+1,
\]
and therefore
\[
  T_a(x)=a+\frac{a^2}{x}
        =a+(n+1)a
        =(n+2)a
        =z\in A_p.
\]
The resulting ordered pairs \((x,a)\) are distinct: the pair determines
\(a/x=n+1\), and the residues \(1,2,\ldots,p-2\) are distinct modulo
\(p\).  Thus at least \(p-2\) pairs are counted.
\end{proof}

\section{Cauchy--Schwarz and affine linearisation}

For \(x\in A\), write
\[
  R(x)=|\{a\in A:T_a(x)\in A\}|.
\]
Then \(\cN(A)=\sum_{x\in A}R(x)\), so Cauchy--Schwarz gives
\begin{equation}\label{eq:cauchy}
  \cN(A)^2
  \leq m\sum_{x\in A}R(x)^2
  =m\sum_{a,b\in A}M(a,b),
\end{equation}
where
\[
  M(a,b)
  =|\{x\in A:T_a(x)\in A,\ T_b(x)\in A\}|.
\]

The map \(T_a\) is injective on \(\F_p^\times\), and its inverse on
\(\F_p\setminus\{a\}\) is
\[
  T_a^{-1}(u)=\frac{a^2}{u-a}.
\]
Consequently,
\begin{align*}
  T_b(T_a^{-1}(u))
  &=b+\frac{b^2}{a^2/(u-a)}\\
  &=\left(\frac ba\right)^2u+b-\frac{b^2}{a}.
\end{align*}
Define the affine line
\begin{equation}\label{eq:line}
  \ell_{a,b}:\quad
  v=\left(\frac ba\right)^2u+b-\frac{b^2}{a}.
\end{equation}
For each \(a,b\in A\), the substitution \(u=T_a(x)\) is injective in
\(x\), and hence
\begin{equation}\label{eq:M-incidence}
  M(a,b)
  \leq
  |\{u\in A:(u,\ell_{a,b}(u))\in A\times A\}|.
\end{equation}
The following elementary multiplicity bound is the other structural
input.

\begin{lemma}\label{lem:multiplicity}
Let \(p\) be odd and let \(A\subseteq\F_p^\times\).
\begin{enumerate}[label=\textup{(\roman*)},leftmargin=2.1em]
\item \(\ell_{a,b}\) is the identity line \(v=u\) if and only if
      \(a=b\).
\item Every nonidentity affine line has at most two representations of
      the form \(\ell_{a,b}\) with \((a,b)\in A^2\).
\end{enumerate}
\end{lemma}

\begin{proof}
Put \(r=b/a\).  Then \eqref{eq:line} can be written as
\begin{equation}\label{eq:line-ratio}
  \ell_{a,b}:\qquad v=r^2u+ar(1-r).
\end{equation}
Thus \(\ell_{a,b}\) is the identity line precisely when
\[
  r^2=1
  \qquad\text{and}\qquad
  ar(1-r)=0.
\]
Because \(a,r\neq0\) and \(p\) is odd, the second equality excludes
\(r=-1\); hence \(r=1\), which is equivalent to \(a=b\).  Conversely,
\(a=b\) gives \(r=1\) and therefore \(v=u\).

For the second assertion, fix a nonidentity line
\(\ell:v=\sigma u+\tau\).  Every representation
\(\ell=\ell_{a,b}\) determines a ratio \(r=b/a\) satisfying
\begin{equation}\label{eq:ratio-recovery}
  r^2=\sigma,
  \qquad
  ar(1-r)=\tau.
\end{equation}
The first equation has at most two solutions \(r\in\F_p\).  Moreover,
no such solution is \(r=1\), since that would represent the identity
line.  Once \(r\) is fixed, the second equation uniquely determines
\[
  a=\frac{\tau}{r(1-r)}
  \qquad\text{and then}\qquad
  b=ar.
\]
Consequently each of the at most two square roots of the slope gives at
most one ordered parameter pair \((a,b)\).  Hence \(\ell\) has at most
two representations of the required form.
\end{proof}

Let \(\cL_A\) be the set of distinct nonidentity lines occurring among
\(\ell_{a,b}\) with \(a,b\in A\).  From \eqref{eq:M-incidence} and
\cref{lem:multiplicity},
\begin{equation}\label{eq:second-moment-incidence}
  \sum_{a,b\in A}M(a,b)
  \leq m^2+2\I(A\times A,\cL_A).
\end{equation}
Indeed, the \(m\) diagonal pairs \(a=b\) contribute at most \(m\) each,
and every nonidentity line occurs with multiplicity at most two.

\section{Proof of the main theorem}

\begin{proof}[Proof of \cref{thm:main}]
The prime \(p=2\) is harmless, so assume that \(p\) is odd.  Let
\(A=A_p\) and \(m=|A|\).

Let \(c_0\) be the constant from \cref{thm:sdz}.  If
\[
  m^3>c_0p^2,
\]
then \(m\gg p^{2/3}\), which is stronger than the desired conclusion.
We may therefore suppose that
\begin{equation}\label{eq:char-condition}
  m^3\leq c_0p^2.
\end{equation}

The set \(\cL_A\) contains at most \(m^2\) lines.  Enlarge it, if
necessary, to a set \(\widetilde{\cL}_A\) of exactly \(m^2\) distinct
affine lines.  This is possible because \(m\leq p-1\), while there are
\(p^2\) nonvertical affine lines over \(\F_p\).  Apply \cref{thm:sdz} with
\[
  X=Y=A,\qquad x=y=m,\qquad N=m^2.
\]
The condition \(xy^2\leq N^3\) is \(m^3\leq m^6\), and the
characteristic condition is exactly \eqref{eq:char-condition}.  Hence
\begin{align}
  \I(A\times A,\cL_A)
  &\leq \I(A\times A,\widetilde{\cL}_A)\notag\\
  &\ll m^{3/4}m^{1/2}(m^2)^{3/4}+m^2\notag\\
  &\ll m^{11/4}.
  \label{eq:incidence-final}
\end{align}
Combining \eqref{eq:cauchy}, \eqref{eq:second-moment-incidence} and
\eqref{eq:incidence-final}, we obtain
\[
  \cN(A)^2
  \ll m\bigl(m^2+m^{11/4}\bigr)
  \ll m^{15/4}.
\]
On the other hand, \cref{lem:many-transitions} gives
\(\cN(A)\geq p-2\).  Therefore
\[
  (p-2)^2\ll m^{15/4},
\]
and consequently
\[
  m\gg p^{8/15}.
\]
After adjusting the absolute constant, this also covers the finitely
many small primes.
\end{proof}

\begin{remark}
The exponent \(8/15\) comes directly from the Cartesian-product incidence
exponent \(11/4\): the transition count is \(\gg p\), its square is
bounded by
\[
  m\cdot m^{11/4}=m^{15/4},
\]
and hence \(m\gg p^{2/(15/4)}=p^{8/15}\).  Any improvement for the
special line family \eqref{eq:line} would immediately strengthen
\cref{thm:main}.
\end{remark}

\section*{Statement on the use of AI}

GPT-5.6 Pro was used during exploratory work to help develop the
fractional-linear family
\[
  T_a(x)=a+\frac{a^2}{x},
\]
the Cauchy--Schwarz linearisation into the affine lines
\eqref{eq:line}, the line-multiplicity calculation, and the derivation
of the exponent \(8/15\).  It also assisted with algebraic checks,
literature searches, exposition and LaTeX preparation.

\end{document}